\documentclass[psamsfonts,reqno]{amsart}
\usepackage{amssymb,xy,eucal}
\usepackage[usenames]{color}

\usepackage{amscd}

\usepackage{graphics}

\usepackage{mathrsfs}

\xyoption{all}

\theoremstyle{plain}

\newtheorem{lemma}{Lemma}[section]
\newtheorem{theorem}[lemma]{Theorem}

\newtheorem{corollary}[lemma]{Corollary}

\newtheorem*{theorem*}{Theorem}

\newtheorem*{stat}{\name}
\newcommand{\name}{testing}

\theoremstyle{definition}
\newtheorem{definition}[lemma]{Definition}

\theoremstyle{remark}
\newtheorem{remark}[lemma]{Remark}
\newtheorem{notation}[lemma]{Notation}

\newtheorem*{remark*}{Remark}

\newcommand{\qedc}{{\qed}~{\rm Claim~{\theclaim}.}}
\newcommand{\qedsc}{{\qed}~{\rm Claim.}}

\numberwithin{equation}{section}

\newcommand{\pup}[1]{\textup{(}{#1}\textup{)}}

\newcommand{\set}[1]{\{#1\}}
\newcommand{\setm}[2]{\set{#1\mid#2}}

\newcommand{\alg}[1]{\left\langle{#1}\right\rangle}

\newcommand{\Pow}{\mathfrak{P}}

\newcommand{\cC}{\mathcal{C}}

\newcommand{\cF}{\mathcal{F}}
\newcommand{\cG}{\mathcal{G}}
\newcommand{\cH}{\mathcal{H}}

\newcommand{\cR}{\mathcal{R}}
\newcommand{\cS}{\mathcal{S}}
\newcommand{\cT}{\mathcal{T}}
\newcommand{\cV}{\mathcal{V}}

\newcommand{\bA}{\boldsymbol{A}}
\newcommand{\bB}{\boldsymbol{B}}

\newcommand{\bF}{\boldsymbol{F}}

\newcommand{\bH}{\boldsymbol{H}}
\newcommand{\bK}{\boldsymbol{K}}
\newcommand{\bR}{\boldsymbol{R}}
\newcommand{\bS}{\boldsymbol{S}}

\newcommand{\bcH}{\boldsymbol{\mathcal{H}}}

\newcommand{\BA}{\mathbb{A}}

\newcommand{\BU}{\mathbb{U}}

\newcommand{\res}{\mathbin{\restriction}}

\newcommand{\card}[1]{|#1|}

\DeclareMathOperator{\Sub}{Sub}
\DeclareMathOperator{\gra}{graph}

\DeclareMathOperator{\Con}{Con}

\DeclareMathOperator{\Var}{{\bf{Var}}}
\DeclareMathOperator{\QVar}{{\bf{QVar}}}

\DeclareMathOperator{\Hom}{Hom}



\begin{document}

\title{Finite Abelian algebras are dualizable}

\author[P.~Gillibert]{Pierre Gillibert}
\address{Pontificia Universidad  Cat\'olica de Valpara\'iso, Instituto de Matem\'aticas, Valpara\'iso, Chile}
\email{pierre.gillibert@ucv.cl, pgillibert@yahoo.fr}

\date{\today}

\begin{abstract}
A finite algebra $\bA=\alg{A;\cF}$ is \emph{dualizable} if there exists a discrete topological relational structure $\BA=\alg{A;\cG;\cT}$, compatible with $\cF$, such that the canonical evaluation map $e_{\bB}\colon \bB\to \Hom( \Hom(\bB,\bA),\BA)$ is an isomorphism for every~$\bB$ in the quasivariety generated by~$\bA$. Here, $e_{\bB}$ is defined by $e_{\bB}(x)(f)=f(x)$ for all $x\in B$ and all $f\in \Hom(\bB,\bA)$.

We prove that, given a finite congruence-modular Abelian algebra~$\bA$, the set of all relations compatible with~$\bA$, up to a certain arity, \emph{entails} the whole set of all relations compatible with~$\bA$. By using a classical compactness result, we infer that~$\bA$ is dualizable. Moreover we can choose a dualizing alter-ego with only relations of arity $\le 1+\alpha^3$, where $\alpha$ is the largest exponent of a prime in the prime decomposition of $\card{A}$.

This improves Kearnes and Szendrei result that modules are dualizable, and Bentz and Mayr's result that finite modules with constants are dualizable. This also solves a problem stated by Bentz and Mayr in 2013. 
\end{abstract}

\maketitle

\section{Introduction}

A (finite) structure~$\BA$ is an \emph{alter-ego} of an algebra~$\bA$, if both have the same underlying set $A$ and each operation, each relation and each partial operation of~$\BA$ is compatible with all operations of~$\bA$. We also add the discrete topology to~$\BA$.

Under this assumption, given $\bB\in\QVar\bA$ (that is,~$\bB$ is a subalgebra of a power of~$\bA$), the set $\Hom(\bB,\bA)$ of all homomorphisms $\bB\to\bA$ is a closed subset of $\BA^B$ and is stable under each operation of~$\BA$. We denote the corresponding substructure of~$\BA^B$ by~$\bB^*$.

Similarly, given~$\BU$ a closed substructure of~$\BA^X$ (for some nonempty set $X$), the set $\Hom(\BU,\BA)$ of all continuous homomorphisms $\BU\to\BA$ is a subalgebra of~$\bA^U$. We denote the corresponding subalgebra of~$\BA^B$ by~$\BU^+$.

The \emph{evaluation map} is the map $e_B\colon \bB \to \bB^{*+}$ defined by $e_B(x)(f)=f(x)$ for all $x\in B$ and all $f\in\Hom(\bB,\bA)$.

The evaluation map is always an embedding of algebras. The alter-ego~$\BA$ \emph{dualizes}~$\bA$ if $e_B$ is an isomorphism for each $B\in\QVar\bA$. An algebra~$\bA$ is \emph{dualizable} if there exists an alter-ego~$\BA$ that dualizes~$\bA$.

Independently Z\'adori in \cite {Z}, and Davey, Heindorf, and McKenzie in \cite{DHM} prove the following characterization of dualizable algebras in congruence-distributive varieties (see also \cite[10.2.2]{CDBook}).

\begin{theorem*}[Z\'adori, Davey, Heindorf, and McKenzie]
A finite algebra, generating a congruence-distributive variety, is dualizable if and only if it has a near-unanimity term.
\end{theorem*}

Moreover Mar\'oti proves in \cite{Ma} that the existence of a near-unanimity term in a finitely generated variety is decidable. He concludes that the dualizibility of finite algebras generating congruence-distributive varieties is decidable. The equivalent problem for finite algebras generating congruence-modular varieties is still open. However several results are known.

Abelian groups are dualizable (finite discrete case of Pontryagin duality). Quackenbush and Szab\'o prove in \cite{QS} that non-Abelian nilpotent groups are \emph{inherently} non-dualizable (any group containing a non-Abelian nilpotent subgroup is not dualizable). Nickodemus finally proved in \cite{N} that groups with only Abelian Sylow subgroups are dualizable. This give the following characterization of dualizable groups.

\begin{theorem*}[Nickodemus, Quackenbush, and Szab\'o]
A finite group is dualizable if and only if all its Sylow subgroups are Abelian.
\end{theorem*}

Bentz and Mayr prove in \cite{BM} that supernilpotent non-Abelian algebras are inherently non-dualizable. Bentz and Mayr also gave an example of a nilpotent, non-supernilpotent, non-Abelian, dualizable algebra.

A natural question, and the next step in the characterization of dualizable algebras in congruence-modular varieties, is the case of Abelian algebras. Several example are already known. Kearnes and Szendrei prove that finite modules are dualizable. A finite module, with an additional constant operation for each element, is dualizable (unpublished result by Bentz and Mayr)

In this paper we extend both results by proving that every finite Abelian algebra, generating a congruence-modular variety, is dualizable. This solves a problem asked by Bentz and Mayr \cite[Problem 6.1]{BM}. This also characterize supernilpotent dualizable algebras (exactly the Abelian algebras). Note that Kearnes and Szendrei have recently proved in \cite{KS} that Abelian algebras are dualizable. However the proof of this paper yields an explicit bound on the arities of relations in the dualizing structure.

Note that the author, with Bentz and Sequeira in \cite{BGS}, extend the result to prove that finite Abelian algebras are strongly dualizable.

\section{Basic concept}

We refer to \cite{CDBook} for basic definitions and results on duality theory.

Given a set $A$, we denote by $\card A$ the cardinality of $A$, and by $\Pow(A)$ the set of all subsets of $A$. Given a map $f\colon A\to B$ we set $\gra f=\setm{(a,f(a))}{a\in A}$. Note that if $f$ is an operation on $A$ of arity $n$, then $\gra f$ is a relation of arity $n+1$.

Let $\cF$ be a set of operations and relations over a set $A$. We say that $\cF$ \emph{entails} an operation $f$ (resp., a relation $R$), if for each integer $n>0$, each map $p\colon A^n\to A$ compatible with each element of $\cF$ is also compatible with $f$ (resp. $R$). Let $\cG$ be a set of operations and relations. We say that $\cF$ \emph{entails} $\cG$ if $\cF$ entails each relations and each operations in $\cG$.

In this paper we only use the following classical entailment results.
\begin{lemma}\label{L:entailementFaciles}
Let $A$ be a set. 
\begin{enumerate}
\item Let $(R_i)_{i\in I}$ be a family of $k$-ary relations on $A$. Then $\setm{R_i}{i\in I}$ entails $\bigcap_{i\in I}R_i$.
\item Let $\cF$ be a family of operations on $A$. Let $R$ be a $k$-ary relation on $A$ and let $p_1,\dots,p_k\colon A^n\to A$ be terms in $\cF$. We consider the $n$-ary relation on $A$ defined by 
\begin{equation*}
S=\setm{\vec x\in A^n}{ (p_1(\vec x),p_2(\vec x),\dots,p_k(\vec x))\in R}\,.
\end{equation*}
Then $\set{R}\cup\cF$ entails $S$.
\item Let $R$ be an $n$-ary relation. Set $S=\setm{(x_1,\dots,x_n,x_n)}{(x_1,\dots,x_n)\in R}$. Then $\set{S}$ entails $R$.
\item Let $f\colon A^n\to A$ then $\set{\gra f}$ entails $f$.
\end{enumerate}
\end{lemma}

We also use the classical brute-force and compactness argument summarized here.

\begin{theorem}\label{T:entaildualise}
Let~$\bA$ be a finite algebra. Let $\cF$ be a finite set of operations and relations compatible with~$\bA$. Assume that $\cF$ entails each relation compatible with~$\bA$. Then $\alg{A;\bF}$ dualizes~$\bA$.
\end{theorem}

We only consider Abelian algebras generating congruence-modular varieties.

\begin{definition}
A term of an algebra~$\bA$ is \emph{affine} if there is an Abelian group $\alg{A;+,-,0}$. Such that
\begin{enumerate}
\item $t(x,y,z)=x-y+z$ for all $x,y,z\in A$.
\item $t$ is compatible with all terms of~$\bA$. That is, $f(\vec x-\vec y+\vec z)=f(\vec x)-f(\vec y)+f(\vec z)$ for all $n$-ary terms $f$ of~$\bA$ and all $\vec x,\vec y,\vec z$ in $A^n$.
\end{enumerate}
An algebra is \emph{affine} if it has an affine term.
\end{definition}

\begin{remark}\label{R:AlgebreAffine}
The group operations $+$, $-$ and the constant $0$ are not in general terms of~$\bA$. These operations might not be compatible with operations of~$\bA$, and might not be compatible with morphisms.

The Abelian group yielding an affine structure is not unique. If $c$ is any element of $A$ then we can assume that $c$ is the neutral element of $+$. In fact $(x,y)\mapsto x+^c y= t(x,c,y)$ is an Abelian group operation (cf. \cite[Lemma~5.6]{FMBook}). However the term map, induced by~$t$, is unique.
\end{remark}

\begin{lemma}\label{L:TermesAffines}
Let $t(x,y,z)$ be an affine term. Let $u_1,\dots,u_n$ be integers such that $\sum_{k=1}^n u_k=1$, then $\sum_{k=1}^n u_kx_k$ is a term in $t$. Conversely any term in $t$ is equivalent to a term of this form.
\end{lemma}

The following characterization of Abelian algebras in congruence-modular varieties is due to Herrmann \cite{AA}, a complete proof is given in \cite[Corollary~5.9]{FMBook}.
\begin{theorem}[Herrmann]
The Abelian algebras generating congruence-modular varieties are the affine algebras.
\end{theorem}

In the sequel of the paper, whenever we have an congruence-modular variety of Abelian algebras, we fix a term~$t$ giving each algebra an affine structure. We also implicitly fix an Abelian group operation $+$ for each considered Abelian algebra (note that we only consider Abelian algebras in congruence-modular varieties).

\section{Congruences induced by subalgebras}

In this section, we construct a correspondence between subalgebras of an Abelian algebra, and some of its congruences.

Most results are generalizations of the corresponding result for Abelian groups. As there is no natural origin/zero element in an Abelian algebra the formulations and proofs are a little different.

We denote by $\Sub\bA$ the set of all underlying set of (nonempty) subalgebras of~$\bA$. Given $B\in\Sub\bA$, we denote by~$\bB$ the corresponding subalgebra of~$\bA$.

\begin{definition}
Let~$\bA$ be an Abelian algebra. Let $B\in\Sub\bA$. The \emph{congruence generated} by $B$, denoted by $\Theta_B$, is the smallest congruence of~$\bA$ which identifies all pairs of elements in $B$.
\end{definition}

\begin{remark}
Some congruences of~$\bA$ might not be of the form $\Theta_B$ with $B\in\Sub\bA$. For example, if~$\bA$ has at least two different constant operations, then the identity congruence of~$\bA$ cannot be written as $\Theta_B$.

We might have $\Theta_B=\Theta_C$ with $B\not=C$ in $\Sub\bA$. For example, if the only operation of~$\bA$ is~$t$, and~$\bA$ has a nontrivial subalgebra $B$, then taking $c\in A\setminus B$, we can see that $C=c+B=\setm{c+b}{b\in B}\in\Sub\bA$ and $B\not=C$, however $\Theta_B=\Theta_C$.
\end{remark}

In the following lemma, we know that there always exists such a term~$t$. The equation \eqref{E:congruenceforall} and \eqref{E:congruenceexists} could be taken for definition of $\Theta_B$. In particular, we see that it does not depend on the choice of $t$.

\begin{lemma}\label{L:congruencefromsubspace}
Let~$\bA$ be an Abelian algebra. Let $t$ be an affine term of~$\bA$. Let $B\in\Sub\bA$. The following equalities hold
\begin{align}
\Theta_B&=\setm{(x,y)\in A}{\forall b\in B\,,t(x,y,b)\in B}\label{E:congruenceforall}\\
&=\setm{(x,y)\in A}{\exists b\in B\,,t(x,y,b)\in B}\label{E:congruenceexists}\,.
\end{align}
\end{lemma}

\begin{proof}
We fix an Abelian group structure on $A$ such that $t(x,y,z)=x-y+z$, for all $x,y,z\in A$. We set:
\begin{eqnarray}
\Theta'_B=\setm{(x,y)\in A}{\forall b\in B\,,t(x,y,b)\in B}\,,\\
\Theta''_B=\setm{(x,y)\in A}{\exists b\in B\,,t(x,y,b)\in B}\,.
\end{eqnarray}

Let $(x,y)\in\Theta''_B$. Let $b'\in B$ such that $t(x,y,b')\in B$. Let $b\in B$. As $B$ is stable under~$t$, it follows that $t(t(x,y,b'),b',b)\in B$, however:
\begin{equation*}
t(t(x,y,b'),b',b)=x-y+b'-b'+b=x-y+b=t(x,y,b)\,.
\end{equation*}
So $t(x,y,b)\in B$ for all $b\in B$, that is, $(x,y)\in \Theta'_B$. Hence $\Theta''_B\subseteq\Theta'_B$. The other containment follows from the fact that $B$ is not empty.

We now prove that $\Theta'_B$ is a congruence of~$\bA$. Let $x\in A$, let $b\in B$, then $t(x,x,b)=b\in B$, therefore $(x,x)\in\Theta'_B$.

Let $(x,y)\in\Theta'_B$. Let $b\in B$. Set $b'=t(x,y,b)$. Then $b'\in B$, moreover the following equalities hold
\begin{equation}
t(y,x,b')=t(y,x,t(x,y,b))=y-x+x-y+b=b.
\end{equation}
As $b\in B$, it follows that $(y,x)\in \Theta'_B$.

Let $x,y,z$ in $A$ such that $(x,y)\in\Theta'_B$ and $(y,z)\in\Theta'_B$. Let $b\in B$, then $t(x,y,b)\in B$ and  $t(y,z,b)\in B$. Therefore $t(t(x,y,b),b,t(y,z,b))\in B$. Moreover the following equalities hold:
\begin{equation}
t(t(x,y,b),b,t(y,z,b))=x-y+b-b+y-z+b = x-z+b=t(x,z,b)\,.
\end{equation}
Therefore $t(x,z,b)\in B$, it follows that $(x,z)\in\Theta'_B$.

Let $\lambda$ be an $n$-ary operation of~$\bA$. Let $x_1,\dots,x_n,y_1,\dots, y_n$ in $A$ such that $(x_i,y_i)\in\Theta'_B$ for all $1\le i\le n$. Let $b\in B$, we have
\begin{equation}
t(x_i,y_i,b)\in B\,,\quad\text{for all $1\le i\le n$.}
\end{equation}
However~$t$ preserves $\lambda$ and $B$ is stable under $\lambda$, hence:
\begin{equation}
t(\lambda(x_1,\dots,x_n),\lambda(y_1,\dots,x_n),\lambda(b,\dots,b))\in B\,.
\end{equation}
Moreover $\lambda(b,\dots,b)\in B$, therefore $(\lambda(x_1,\dots,x_n),\lambda(y_1,\dots,x_n))\in\Theta'_B$.

Hence $\Theta'_B$ is a congruence of~$\bA$. Moreover, given $x,y\in B$, note that $t(x,y,y)=x\in B$, hence $(x,y)\in \Theta'_B$. Therefore $\Theta_B\subseteq\Theta'_B$.

Let $(x,y)\in\Theta'_B$. Let $b\in B$, hence $t(x,y,b)\in B$. Therefore $b$ and $t(x,y,b)$ are identified by $\Theta_B$, hence $t(t(x,y,b),b,y)$ and $t(b,b,y)$ are identified by $\Theta_B$. However $t(t(x,y,b),b,y)=x-y+b-b+y=x$ and $t(b,b,y)=b-b+y=y$, thus $(x,y)\in\Theta_B$. Therefore $\Theta_B=\Theta'_B$.
\end{proof}

Given a congruence $\alpha$ of~$\bA$, there might not exist a corresponding space $B$ such that $\alpha=\Theta_B$, however if $\alpha\supseteq\Theta_B$, we can assign a set $C(\alpha,B)=\setm{x\in A}{\forall b\in B\,,(x,b)\in\alpha}$.

\begin{lemma}\label{L:subspacefromcongruence}
Let~$\bA$ be an Abelian algebra. Let $B\in\Sub\bA$. Let $\alpha$ be a congruence of~$\bA$ containing $\Theta_B$. Then $B\subseteq C(\alpha,B)\in\Sub\bA$, and:
\begin{equation}
C(\alpha,B)=\setm{x\in A}{\exists b'\in B\,,(x,b')\in\alpha}\,.
\end{equation}
\end{lemma}

\begin{proof}
Let $x\in A$ and $b'\in B$ such that $(x,b')\in\alpha$. Let $b\in B$. Note that $(b',b)\in\Theta_B\subseteq\alpha$. Hence $(x,b)\in\alpha$ for all $b\in B$, that is, $x\in C(\alpha,B)$. The other containment is immediate as $B$ is not empty.

Let $b\in B$. As $(b,b)\in\alpha$, it follows that $b\in C(\alpha,B)$, therefore $B\subseteq C(\alpha,B)$.

Let $\lambda$ be an $n$-ary operation of~$\bA$. Let $x_1,\dots,x_n$ in $C(\alpha,B)$. Let $b\in B$. Note that $(x_i,b)\in \alpha$ for all $1\le i\le n$, therefore $(\lambda(x_1,\dots,x_n),\lambda(b,\dots,b))\in\alpha$. However $\lambda(b,\dots,b)\in B$, hence $\lambda(x_1,\dots,x_n)\in C(\alpha,B)$.
\end{proof}

\begin{lemma}\label{L:bijcongruencesubspace}
Let~$\bA$ be an Abelian algebra. Let $B\in\Sub\bA$. The following statements hold:
\begin{enumerate}
\item Let $\alpha\supseteq\Theta_B$ be a congruence of~$\bA$. Then $\Theta_{C(\alpha,B)}=\alpha$.

\item Let $X\supseteq B$ in $\Sub\bA$. Then $C(\Theta_X,B)=X$
\end{enumerate}
\end{lemma}

\begin{proof}
We fix a group operation $+$ on $A$, such that~$t$, defined by $t(x,y,z)=x-y+z$ is a term of~$\bA$.

Let $x,y$ in $A$. Let $b\in B$. Assume that $(x,y)\in\alpha$, note that $(y,y)\in\alpha$ and $(b,b)\in\alpha$, hence, form the compatibility of $\alpha$ with~$t$, it follows that $(t(x,y,b),b)=(t(x,y,b),t(y,y,b))\in\alpha$. Conversely if $(t(x,y,b),b)\in\alpha$, then
\begin{equation*}
(x,y)=(t(t(x,y,b),b,y) , t(b,b,y))\in\alpha\,.
\end{equation*}
This proves the following equivalence:
\begin{equation}\label{E:L:bij1}
(x,y)\in\alpha\Longleftrightarrow (t(x,y,b),b)\in\alpha\,.
\end{equation}

The following equivalences hold:
\begin{align*}
(x,y)\in\Theta_{C(\alpha,B)} &\Longleftrightarrow t(x,y,b)\in C(\alpha,B) && \text{By Lemma~\ref{L:congruencefromsubspace}.}\\
&\Longleftrightarrow (t(x,y,b),b)\in\alpha &&\text{By Lemma~\ref{L:subspacefromcongruence}.}\\
&\Longleftrightarrow (x,y)\in\alpha &&\text{By \eqref{E:L:bij1}.}
\end{align*}
So $\Theta_{C(\alpha,B)}=\alpha$, that is, $(1)$ holds.

Fix $X\in\Sub\bA$ such that $X\supseteq B$. Let $x$ in $A$, let $b\in B$.
\begin{equation*}
x\in C(\Theta_X,B) \Longleftrightarrow (x,b)\in\Theta_X \Longleftrightarrow x=t(x,b,b)\in X\,.
\end{equation*}
That is, $(2)$ holds.
\end{proof}

\begin{theorem}\label{T:bijcongruencesubspace}
Let~$\bA$ be an Abelian algebra. Let $B\in\Sub\bA$. Then $X\mapsto\Theta_X$ and $\alpha\mapsto C(\alpha,B)$ are mutually inverse isomorphisms of lattices between subalgebras of~$\bA$ containing $B$ and congruences of~$\bA$ containing $\Theta_B$.
\end{theorem}

\begin{proof}
Denote by $\cS=\setm{X\in\Sub\bA}{X\supseteq B}$ the lattice of all subalgebras of~$\bA$ containing $B$, and by $\cC=\setm{\alpha\in\Con\bA}{\alpha\supseteq\Theta_B}$ the lattice of all congruences of~$\bA$ containing $\Theta_B$.

We can see from the definition, that if $X\supseteq Y$, then $\Theta_X\supseteq\Theta_Y$. In particular $X\mapsto\Theta_X$ is an isotone map from $\cS$ to $\cC$.

Similarly if $\alpha\subseteq\beta$ then $C(\alpha,B)\subseteq C(\beta,B)$. The conclusion follows from Lemma~\ref{L:bijcongruencesubspace}.
\end{proof}

It follows that a quotient of an Abelian algebra by a meet-irreducible subalgebra is subdirectly irreducible. That is, the following statement holds.

\begin{corollary}\label{C:quotientSI}
Let~$\bA$ be an Abelian algebra. Let $B\in\Sub\bA$. If $B$ is a meet-irreducible, then $A/\Theta_B$ is subdirectly irreducible.
\end{corollary}

The following corollary expresses that a completely meet-irreducible subalgebra of an Abelian algebra is a ``kernel'' of a morphism into a subdirectly irreducible algebra.

\begin{corollary}\label{C:quotient}

Let~$\bA$ be an Abelian algebra, let $B\in\Sub\bA$ be completely meet-irreducible. Then there exists a subdirectly irreducible algebra~$\bS$ in the variety generated by~$\bA$, a morphism of algebras $f\colon \bA\to \bS$ and an element $c\in S$ such that:
\begin{enumerate}
\item $B=\setm{x\in A}{f(x)=c}$.
\item $c$ is preserved by all operations of~$\bS$, equivalently $\set{c}$ is the underlying set of a subalgebra of~$\bS$.
\end{enumerate}
\end{corollary}

\begin{proof}
Set $\bS=\bA/\Theta_B$. Denote by $f\colon\bA\to\bS$ the canonical projection. By Corollary~\ref{C:quotientSI} the algebra~$\bS$ is subdirectly irreducible. Let $b\in B$, note that $f(b)=b/\Theta_B=\setm{x\in A}{(x,b)\in\Theta_B}=C(\Theta_B,B)=B$, we denote by $c$ this element of~$\bS$. Hence $\setm{x\in A}{f(x)=c}=B$.

Moreover, given $b\in B$ and $\lambda$ an operation of~$\bA$, we have $\lambda(b,\dots,b)\in B$, hence $\lambda(b,\dots,b)/\Theta_B=B=c$, that is, $\lambda(c,\dots,c)=c$.
\end{proof}

\section{Counting morphisms}

\begin{lemma}\label{L:CardinalHomomorphismes}
Let~$\bA$ and~$\bB$ be Abelian algebras in a congruence-modular variety. Assume that $\card{A}=p_1^{\alpha_1}p_2^{\alpha_2}\dots p_k^{\alpha_k}$ and $\card{B}=p_1^{\beta_1}p_2^{\beta_2}\dots p_k^{\beta_k}$, where $p_1,\dots,p_k$ are distinct primes. Then the following statement holds
\begin{equation*}
\card{\Hom(\bA,\bB)}\text{ divides }p_1^{(\alpha_1+1)\beta_1}p_2^{(\alpha_2+1)\beta_2}\dots p_k^{(\alpha_k+1)\beta_k}\,.
\end{equation*}
Moreover, if~$\bA$ and~$\bB$ have a group structure then
\begin{equation*}
\card{\Hom(\bA,\bB)}\text{ divides }p_1^{\alpha_1\beta_1}p_2^{\alpha_2\beta_2}\dots p_k^{\alpha_k\beta_k}\,.
\end{equation*}
\end{lemma}

\begin{proof}
Assume that~$\bA$ and~$\bB$ are Abelian groups. Given $1\le i\le k$, denote by~$\bA_i$ the $p_i$-Sylow subgroup of~$\bA$. That is, $A_i$ is the set of all element of~$\bA$ of order a power of $p_i$. Similarly denote by~$\bB_i$ the $p_i$-Sylow subgroup of~$\bB$. Note that $\card{A_i}=p_i^{\alpha_i}$ and $\card{B_i}=p_i^{\beta_i}$. Let $F_i$ be a generating set of~$\bA_i$ such that $\card{F_i}=\alpha_i$. Given $a\in F_i$ and $f\in\Hom(\bA,\bB)$, the element $f(a)$ of~$\bB$ is of order a power of $p_i$ hence $f(a)\in B_i$.

Therefore the following map is well-defined:
\begin{align*}
\varphi\colon\Hom(\bA,\bB)&\to B_1^{F_1}\times B_2^{F_2}\times\dots\times B_k^{F_k}\\
f&\mapsto (f\res F_1,f\res F_2,\dots,f\res F_k)\,.
\end{align*}
As~$\bA$ is generated by $F_1\cup F_2\cup\dots\cup F_k$, the map $\varphi$ is injective. Moreover $\varphi$ is a morphism of groups. Therefore $\alg{\Hom(\bA,\bB);+}$ is isomorphic to a subgroup of~$\bB_1^{F_1}\times \bB_2^{F_2}\times\dots\times \bB_k^{F_k}$. So $\card{\Hom(\bA,\bB)}$ divides $\card{B_1^{F_1}\times B_2^{F_2}\times\dots\times B_k^{F_k}}$. Moreover the following equalities hold
\begin{align*}
\card{B_1^{F_1}\times B_2^{F_2}\times\dots\times B_k^{F_k}}
&=(p_1^{\beta_1})^{\card{F_1}}\times (p_2^{\beta_2})^{\card{F_2}}\times \dots\times (p_k^{\beta_k})^{\card{F_k}}\\
&=(p_1^{\alpha_1\beta_1})\times (p_2^{\alpha_2\beta_2})\times \dots\times (p_k^{\alpha_k\beta_k})\,.
\end{align*}

We now only assume that~$\bA$ and~$\bB$ are Abelian algebras. Note that $\Hom(\bA,\bB)$ is the underlying set of a sub-algebra of $\alg{B;t}^A$. Set $H=\Hom(\bA,\bB)$ and $\bH=\alg{H;t}$.

Pick $a\in A$. We can assume that the neutral element of the Abelian group operation $+$ on $A$, inducing the affine structure of~$\bA$, is $a$. That is, $x+y=t(x,a,y)$ for all $x,y\in A$. Similarly, pick $b\in B$, we can assume that $x+y=t(x,b,y)$ for all $x,y\in B$.

We consider $K=\Hom(\alg{A;+},\alg{B;+})$ and $\bK=\alg{K;t}$. That is, we only consider the affine structure on the set of group morphisms $\alg{A;+}\to \alg{B;+}$. 

Given $f\in\Hom(\bA,\bB)$, we consider $\psi(f)\colon A\to B$, $x\mapsto f(x)-f(a)$. The following equalities hold
\begin{align*}
\psi(f)(x+y)&=f(x+y)-f(a)\\
&=f(x-a+y)-f(a)\\
&=f(x)-f(a)+f(y)-f(a)\\
&=\psi(f)(x)+\psi(f)(y)\,.
\end{align*}
Hence $\psi(f)$ is a morphism of groups, hence $\psi\in K$. Let $f,g,h\in H$. Let $x\in A$. The following equalities hold
\begin{align*}
\psi(f-g+h)(x)&=(f-g+h)(x)-t(f-g+h)(a)\\
&=f(x)-g(x)+h(x)-(f(a)-g(a)+h(a))\\
&=f(x)-f(a)-(g(x)-g(a))+h(x)-h(a)\\
&=\psi(f)(x)-\psi(g)(x)+\psi(h)(x)\,.
\end{align*}
That is, $\psi(t(f,g,h))=\psi(f-g+h)=\psi(f)-\psi(g)+\psi(h)=t(\psi(f),\psi(g),\psi(h))$. Therefore $\psi\colon\bH\to\bK$ is a morphism. Moreover $\bH\to \alg{B;t}$, $f\mapsto f(a)$ is a morphism. Thus the following map defines a morphism
\begin{align*}
\sigma\colon\bH&\to\bK\times\alg{B;t}\\
f&\mapsto (\psi(f),f(a))\,.
\end{align*}
Let $f,g\in H$ such that $\sigma(f)=\sigma(g)$. That is, $f(a)=g(a)$ and $\psi(f)=\psi(g)$. Let $x\in A$, note that $\psi(f)(x)=\psi(g)(x)$, hence $f(x)-f(a)=g(x)-g(a)$, so $f(x)=g(x)$, hence $f=g$. Therefore $\sigma$ is an embedding. Thus $\card{H}$ divides $\card{K\times B}$. Moreover from the group case $\card K$ divides $p_1^{\alpha_1\beta_1}p_2^{\alpha_2\beta_2}\dots p_k^{\alpha_k\beta_k}$. Therefore $\card H$ divides 
$p_1^{\alpha_1\beta_1}p_2^{\alpha_2\beta_2}\dots p_k^{\alpha_k\beta_k} p_1^{\beta_1}p_2^{\beta_2}\dots p_k^{\beta_k}$.
\end{proof}

\begin{lemma}\label{L:TaileFamilleGeneratrice}
Let~$\bA$ be an Abelian group \pup{resp., an Abelian algebra}. Assume that $\card A=p_1^{\alpha_1}p_2^{\alpha_2}\dots p_k^{\alpha_k}$, where $p_1,\dots,p_k$ are distinct primes. Then~$\bA$ has a family of generators with $\max_{1\le i\le k}(\alpha_i)$ elements \pup{resp., $1+\max_{1\le i\le k}(\alpha_i)$ elements}.
\end{lemma}

\begin{proof}
Set $N=\max_{1\le i\le k}(\alpha_i)$. 

If~$\bA$ is an Abelian group. It follows from the fundamental theorem on finitely generated Abelian groups that~$\bA$ is isomorphic to $\mathbb{Z}_{u_1}\times\dots\times \mathbb{Z}_{u_q}$, where $1<u_1\mid u_2\mid \dots\mid u_q$. Let $i$ such that $p_i\mid u_1$, so $p_i^q \mid u_1\dots u_q=\card{A}=p_1^{\alpha_1}p_2^{\alpha_2}\dots p_k^{\alpha_k}$. Hence $q\le\alpha_i\le N$. Therefore~$\bA$ has a family of generators with $q\le N$ elements.

If~$\bA$ is an Abelian algebra. Denote by $+$ an Abelian group operation on $A$, yielding the affine structure of~$\bA$. Denote by $a_0$ the neutral element of $+$. Let $a_1,\dots,a_N$ be a generating family of $\alg{A;+}$. As $x+y=t(x,a_0,y)$ for all $x,y\in A$, it follows that $a_0,\dots,a_N$ generate~$\bA$.
\end{proof}

Kearnes gives in \cite{K} an upper bound to subdirectly irreducible algebras in a finitely generated, congruence-modular variety satisfying the congruence-extension property, this result applies in particular for Abelian algebras. Indeed Kearnes points that this case corresponds to modules.

\begin{lemma}[Kearnes]
Let~$\bA$ be a finite Abelian algebra, let~$\bS$ be a subdirectly irreducible algebra in $\Var\bA$. Then $\card S$ divides $\card{\Hom(\alg{A;+},\alg{A;+})}$.
\end{lemma}

As an immediate consequence of Kearnes result and of Lemma~\ref{L:CardinalHomomorphismes} we obtain the following corollary.

\begin{corollary}\label{C:CardinalSI}
Let~$\bA$ be an Abelian algebra in a congruence-modular variety. Assume that $\card{A}=p_1^{\alpha_1}p_2^{\alpha_2}\dots p_k^{\alpha_k}$, where $p_1,\dots,p_k$ are distinct primes. Let~$\bS$ be a subdirectly irreducible algebra in $\Var\bA$. Then $\card S$ divides $p_1^{\alpha_1^2}p_2^{\alpha_2^2}\dots p_k^{\alpha_k^2}$.
\end{corollary}

\section{Main}

We want to find a finite set of compatible relations which entails the set of all compatible relations. The following lemma expresses that we can restrict our problem to completely meet-irreducible compatible relations.

\begin{lemma}\label{L:restricitoninfirreductibles}
Let~$\bA$ be a finite Abelian algebra. Let $n$ be a positive integer. Denote by $\cF$ the set of all $n$-ary completely meet-irreducible relations compatible with~$\bA$. Let $R$ be an $n$-ary relation compatible with~$\bA$. Then $\cF$ entails $R$.
\end{lemma}

\begin{proof}
Let~$\bB$ be a subalgebra of~$\bA^n$. Note that $B$ is the intersection of all underlying sets of completely meet-irreducible subalgebras of~$\bA^n$ containing $B$. However this set is finite (as~$\bA$ is finite), and is contained in $\cF$. We can conclude with Lemma~\ref{L:entailementFaciles}(1) that $\cF$ entails $R$.
\end{proof}

As an immediate consequence of Theorem~\ref{T:entaildualise} and Lemma~\ref{L:restricitoninfirreductibles}

\begin{corollary}
Let~$\bA$ be an algebra. Let $\cR$ be a finite set of relations compatible with~$\bA$. If $\cR$ entails each meet-irreducible subalgebras of $A^n$ for each $n\in\mathbb{N}$, then $\alg{A;\cR}$ dualizes~$\bA$.
\end{corollary}

\begin{notation}
Let~$\bA$ and~$\bS$ be Abelian algebras, let $k\colon\bA\to\bS$ be a morphism. 
\begin{enumerate}
\item Set $\cH_k(\bA^2,\bS)=\setm{f\colon\bA^2\to\bS}{\forall x\in A\ ,f(x,x)=k(x)}$.
\item Let $f\colon\bA^2\to\bS$, let $a\in A$. We set $f_a\colon A\to S$, $x\mapsto f(a,x)$.
\end{enumerate}
\end{notation}

We refer to Remark~\ref{R:AlgebreAffine} for the various definitions of $+$.

\begin{lemma}\label{L:groupeidempotent}
Let~$\bA$ and~$\bS$ be Abelian algebras, let $k\colon\bA\to\bS$ be a morphism. The following statements hold:
\begin{enumerate}
\item The set $\cH_k(\bA^2,\bS)$ is the underlying set of a subalgebra of $\alg{\Hom(\bA^2,\bS);t}$.
\item The map $\overline k\colon\bA^2\to\bS$, $(x,y)\mapsto k(y)$, belongs to $\cH_k(\bA^2,\bS)$.
\item Pick $a\in A$. Then the following map
\begin{align*}
\Psi\colon \alg{\cH_k(\bA^2,\bS);+^{\overline k}} &\to \alg{\Hom( \alg{A;+^a},\alg{S;+^{k(a)}}); +^k}\\
f &\mapsto f_a\,,
\end{align*}
is an embedding of groups.
\item Let $j\colon\bA\to\bS$. Denote $\overline j\colon\bA^2\to\bS$, $(x,y)\mapsto j(y)$. Then $\alg{\cH_k(\bA^2,\bS);+^{\overline k}}$ is isomorphic to $\alg{\cH_j(\bA^2,\bS);+^{\overline j}}$.
\end{enumerate}
\end{lemma}

\begin{proof}
Let $f,g,h\in\cH_k(\bA^2,\bS)$. Let $x\in A$, the following equalities hold 
\[
t(f,g,h)(x,x) = t(f(x,x),g(x,x),h(x,x)) = t(k(x),k(x),k(x)) = k(x)\,.
\]
Hence $t(f,g,h)\in\cH_k(\bA^2,\bS)$, that is, $\cH_k(\bA^2,\bS)$ is the underlying set of a subalgebra of $\alg{\Hom(\bA^2,\bS);t}$. Therefore (1) holds.

Set $\overline k\colon\bA^2\to\bS$, $(x,y)\mapsto k(y)$. First note that $\overline k\colon\bA^2\to\bS$ is a morphism. Moreover, given $x\in A$, we have $\overline k(x,x)=k(x)$, hence $\overline k\in\cH_k(\bA^2,\bS)$. That is, (2) holds.

Pick $a\in A$. Note that $f\colon A^2\to S$ and $A\to S$, $x\mapsto a$, preserve $t$, therefore $f_a$ preserves $t$, moreover $f_a(a)=f(a,a)=k(a)$, thus $f_a$ preserves the neutral element. Therefore $\Psi(f)=f_a\colon \alg{A;+^a}\to\alg{S;+^{k(a)}}$ is a morphism of groups.

Let $f,g\in \cH_k(\bA^2,\bS)$. Let $x\in A$. The following equalities hold
\begin{align*}
(f+^{\overline k}g)_a(x) &= (t(f,\overline k,g))_a(x)\,, &&\text{by definition of $+^{\overline k}$.}\\
& = t(f,\overline k,g)(a,x)\,, &&\text{by definition of $(t(f,\overline k,g))_a$.}\\
& = t(f(a,x),\overline k(a,x),g(a,x))\,, &&\text{by definition of $t(f,\overline k,g)$.}\\
& = t(f_a(x),k(x),g_a(x))\,, &&\text{by definition of $\overline k$, $f_a$, and $g_a$.}\\
& = t(f_a,k,g_a)(x)\,, &&\text{by definition of $t(f_a,k,g_a)$.}\\
& = (f_a +^k g_a)(x)\,, &&\text{by definition of $+^k$.}
\end{align*}
Therefore $\Psi(f+^{\overline k}g)= \Psi(f) +^{k}  \Psi(g)$, so $\Psi$ is a morphism of groups.

We now prove that the kernel of $\Psi$ is reduced to $\set{\overline k}$. Let $f\in\cH_k(\bA^2,\bS)$ such that $\Psi(f)=k$, that is
\begin{equation}\label{E:k1}
f(a,x)=k(x)\,,\quad\text{for all $x\in A$.}
\end{equation}
Therefore the following equalities hold
\begin{align*}
f(x,y) &=f( t(a,a,x), t(y,x,x))\\
& = t(f(a,y),f(a,x),f(x,x))\,, &&\text{as $f$ compatible with $t$.}\\
& = t(k(y),k(x),k(x))\,, &&\text{by \eqref{E:k1} and as $f(x,x)=k(x)$.}\\
& = k(y)\\
& = \overline k(x,y)
\end{align*}
So $f=\overline k$. Therefore $\Psi$ is an embedding. Hence (3) holds.

Let $j\colon\bA\to\bS$. Set $\overline j\colon\bA^2\to\bS$, $(x,y)\mapsto j(y)$. We consider
\begin{align*}
\Phi\colon \cH_k(\bA^2,\bS)&\to \cH_j(\bA^2,\bS)\\
f&\mapsto t(f,\overline k,\overline j)\,.
\end{align*}
Given $f\in \cH_k(\bA^2,\bS)$, the map $t(f,\overline k,\overline j)$ is a composition of morphisms, and so is a morphism. Moreover, given $x\in A$, the following equalities hold
\begin{align*}
t(f,\overline k,\overline j)(x,x)= t(f(x,x),\overline k(x,x),\overline j(x,x)) = t(k(x),k(x),j(x))=j(x)\,.
\end{align*}
Therefore $\Phi(f)=t(f,\overline k,\overline j)$ belongs to $\cH_j(\bA^2,\bS)$. Also note that $t$ and constant maps are all compatible with $t$, thus $\Phi$ is compatible with $t$. Moreover $\Phi(\overline k) = t(\overline k,\overline k,\overline j)=\overline j$, that is, $\Phi$ maps the neutral element of $+^{\overline k}$ to the neutral element of $+^{\overline j}$. Therefore $\Phi\colon \alg{\cH_k(\bA^2,\bS);+^{\overline k}}\to \alg{\cH_j(\bA^2,\bS);+^{\overline j}}$ is a morphism of groups.

Similarly $\alg{\cH_j(\bA^2,\bS);+^{\overline j}}\to \alg{\cH_k(\bA^2,\bS);+^{\overline k}}$, $f\mapsto t(f,\overline j,\overline k)$ is a morphism, which is the inverse of $\Phi$. Thus (4) holds.
\end{proof}

\begin{remark}
From Lemma~\ref{L:groupeidempotent}(4) we see that the group $\alg{\cH_k(\bA^2,\bS);+^{\overline k}}$ does not depend on the choice of $k$. This group is well-defined (up to isomorphism) if there exists at least one morphism $k\colon\bA\to\bS$, in this case we denote this group $\bcH(\bA^2,\bS)$.
\end{remark}

The following corollary is a consequence of Lemma~\ref{L:groupeidempotent}(3) and Lemma~\ref{L:TaileFamilleGeneratrice}.

\begin{corollary}\label{C:cardinalHAS}
Let~$\bA$ and~$\bS$ be Abelian algebras. Assume that $\card A=p_1^{\alpha_1}\dots p_k^{\alpha_k}$ and $\card S= p_1^{\beta_1}\dots p_k^{\beta_k}$, where $p_1,\dots,p_k$ are distinct primes. Then $\card{\cH(\bA^2,\bS)}$ divides $p_1^{\alpha_1\beta_1}\dots p_k^{\alpha_k\beta_k}$, in particular $\bcH(\bA^2,\bS)$ has a generating family with $\max_{1\le i\le k}(\alpha_i\beta_i)$ elements.
\end{corollary}

Morphisms of a large enough power of~$\bA$ to~$\bS$ can be factorized through a smaller power, moreover the factorization is through a morphism that can be expressed with~$t$. Note that, if there is a morphism $\bA^n\to\bS$, then $\bcH(\bA^2,\bS)$ is well-defined.

\begin{lemma}\label{L:factoriser}
Let~$\bA$ and~$\bS$ be algebras in a congruence-modular variety of Abelian algebras. Let $N$ be positive integer such that $\bcH(\bA^2,\bS)$ has a family of generators with $N$ elements. Let $n\ge 1$ be an integer. Let $f\colon\bA^n\to\bS$ be a morphism. There exist a morphism $g\colon \bA^{N+1}\to\bS$ and terms $p_1,\dots,p_{N+1}\colon\bA^n\to\bA$ in~$t$ such that $f(\vec x)= g(p_1(\vec x),p_2(\vec x),\dots,p_{N+1}(\vec x))$, for all $\vec x\in A^n$.
\end{lemma}

\begin{proof}
Let $h_1,\dots,h_N$ be a family generating $\bcH(\bA^2,\bS)$. Given $1\le i\le n$, we consider the morphism $f_i$ defined by
\begin{align*}
f_i\colon \bA^2&\to\bS\\
(x,y)&\mapsto f(y,\dots,y,x,y,\dots,y)\,,\text{where $x$ appears at position $i$.}
\end{align*}
We also define $k\colon\bA\to\bS$, $x\mapsto f(x,\dots,x)$, and $\overline k\colon\bA^2\to\bS$, $(x,y)\mapsto k(y)$.

We can assume that $\bcH(\bA^2,\bS)=\alg{\cH_k(\bA^2,\bS);+^{\overline k}}$. Note that $f_i(x,x)=k(x)$ for all $x\in A$, so $f_i\in\cH_k(\bA^2,\bS)$ for all $1\le i\le n$. Thus there are integers $u_1^i,\dots,u_N^i$ such that
\begin{equation*}
f_i=\sum_{j=1}^N u_j^ih_j\,,\quad\text{for all $1\le i\le n$.}
\end{equation*}
Note that the sum might not be a term in $t$. However $\overline k$ is the neutral element, hence
\begin{equation}\label{E:sumf}
f_i=\sum_{j=1}^N(u_j^ih_j-u_j^i\overline k) + \overline k\,.
\end{equation}
By Lemma~\ref{L:TermesAffines}, the sum in \eqref{E:sumf} is a term in $t$, therefore the following equality holds for all $x,y\in A$ and all $1\le i\le n$.
\begin{equation*}
f_i(x,y) = \sum_{j=1}^N (u_j^ih_j(x,y) - u_j^i\overline k(x,y)) + \overline k(x,y)\,.
\end{equation*}
Note that $\overline k(x,y)=k(y)=f_j(y,y)=h_j(y,y)$. Therefore
\begin{equation}\label{E:fiAvechj}
f_i(x,y) - f_i(y,y) = \sum_{j=1}^N (u_j^ih_j(x,y) - u_j^i h_j(y,y)) \,.
\end{equation}

Given $1\le i\le n$, $z\in A$, and $\vec x\in A^n$, the following equalities hold
\begin{align*}
f(\vec x)-(f_i(x_i,z)-f_i(z,z))&=f(x_1,\dots,x_n)-f(z,\dots,z,x_i,z,\dots,z)+f(z,\dots,z)\\
&=f(x_1,\dots,x_{i-1},z,x_{i+1},\dots,x_n)
\end{align*}
Inductively we deduce that
\begin{equation*}
f(\vec x)-\sum_{i=1}^n(f_i(x_i,z)-f_i(z,z)) = f(z,\dots,z)=k(z)\,.
\end{equation*}
In particular for $z=x_1$ we obtain
\begin{equation}\label{E:calculerfAvecfi}
f(\vec x)=\sum_{i=1}^n(f_i(x_i,x_1)-f_i(x_1,x_1)) + k(x_1)\,,\quad\text{for all $\vec x\in A^n$.}
\end{equation}

Define $p_{N+1}\colon \bA^n\to \bA$, $\vec x\mapsto x_1$. Given $1\le j\le N$, we consider $p_j$ defined by
\begin{equation}
p_j(x_1,\dots,x_n)=\sum_{i=1}^n\left(u_j^i x_i- u_j^i x_1\right)+x_1
\end{equation}
It follows from Lemma~\ref{L:TermesAffines} that $p_j$ is a term in $t$, in particular $p_j$ is a morphism. Given $\vec x\in A^n$ and $z\in A$, the following equalities hold
\begin{align*}
h_j(p_j(\vec x),z) &= h_j\left(\sum_{i=1}^n\big(u_j^i x_i- u_j^i x_1\big)+x_1,z\right)\\
&= h_j\left(\sum_{i=1}^n\big(u_j^i x_i- u_j^i x_1\big)+x_1,\sum_{i=1}^n\big(u_j^iz-u_j^iz\big)+z\right)\\
&=\sum_{i=1}^n\big(u_j^i h_j(x_i,z)- u_j^i h_j(x_1,z)\big) +  h_j(x_1,z)
\end{align*}
Therefore the following equality holds
\begin{equation}\label{E:hjpj}
h_j(p_j(\vec x),z) - h_j(x_1,z) = \sum_{i=1}^n\big(u_j^i h_j(x_i,z)- u_j^i h_j(x_1,z)\big) \,.
\end{equation}

Define $g$ by
\begin{align*}
g\colon \bA^N\times\bA&\to\bS\\
(y_1,\dots,y_N,z)&\mapsto\sum_{j=1}^N\left(h_j(y_j,z) - h_j(z,z)\right) + k(z)\,.
\end{align*}
Note that $g$ is a morphism. Denote by $p\colon\bA^n\to\bA^{N+1}$, $\vec x\mapsto (p_1(\vec x),\dots,p_{N+1}(\vec x))$.

The following equalities hold
\begin{align*}
g(p(\vec x)) &=g(p_1(\vec x),\dots,p_N(\vec x),x_1)\,,&&\text{as $p_{N+1}(\vec x)=x_1$.}\\
&= \sum_{j=1}^N\bigg(h_j(p_j(\vec x),x_1) - h_j(x_1,x_1)\bigg) + k(x_1)\,,&&\text{by definition of $g$.}\\
&= \sum_{j=1}^N\left(\sum_{i=1}^n\left(u_j^i h_j(x_i,x_1)- u_j^i h_j(x_1,x_1)\right)\right) + k(x_1)\,, &&\text{by \eqref{E:hjpj}.}\\
&=\sum_{i=1}^n\left( \sum_{j=1}^N (u_j^i h_j(x_i,x_1)- u_j^i h_j(x_1,x_1))\right)+ k(x_1)\,,&&\text{permuting the sum.}\\
&= \sum_{i=1}^n\left( f_i(x_i,x_1)- f_i(x_1,x_1)\right)+ k(x_1)\,, &&\text{by \eqref{E:fiAvechj}.}\\
&= f(x_1,\dots,x_n)\,, &&\text{by \eqref{E:calculerfAvecfi}.}
\end{align*}
That is, $f(x_1,\dots,x_n)=g(p_1(\vec x),p_2(\vec x),\dots,p_{N+1}(\vec x))$.
\end{proof}

\begin{lemma}\label{L:PrincipalEntailement}
Let~$\bA$ be an Abelian algebra. Let $N$ be a positive integer such that, for all subdirectly irreducible algebras~$\bS$ in $\Var\bA$, the group $\bcH(\bA^2,\bS)$ has a generating family with $N$ elements. We denote by $\cR_n$ the set of $n$-ary relation compatible with~$\bA$. Then $\cR_{N+1}\cup\set{t}$ entails $\cR=\bigcup_{n\ge 1}\cR_n$.
\end{lemma}

\begin{proof}
Let $n\ge 1$. Let $R$ be a $n$-ary completely meet-irreducible compatible relation with~$\bA$. That is, $R$ is the underlying set of~$\bR$ a completely meet-irreducible subalgebra of~$\bA^{n}$.

Therefore it follows from Corollary~\ref{C:quotient} that there exist a subdirectly irreducible algebra~$\bS\in\Var\bA$, a morphism of algebras $f\colon \bA^{n}\to \bS$, and an element $c\in S$ such that $\set c$ is the underlying set of a subalgebra of~$\bS$, and
\begin{equation}\label{E:Restpreimage}
R=f^{-1}(\set c)=\setm{\vec x\in A^{n+1}}{f(\vec x)=c}
\end{equation}

From Lemma~\ref{L:factoriser} there is a morphism $g\colon\bA^N\to\bS$ and terms $p_1,\dots,p_{N+1}\colon\bA^{n}\to\bA$ in~$t$ such that
\begin{equation}\label{E:factor}
f(\vec x) = g(p_1(\vec x),p_2(\vec x),\dots,p_{N+1}(\vec x))\,,\quad\text{for all $\vec x\in A^{n+1}$.}
\end{equation}
Set $B=g^{-1}(\set{c})$, as $g\colon\bA^n\to\bS$ is a morphism and $\set{c}$ is the underlying set of a subalgebra of~$\bS$, it follows that $B$ is an $n$-ary relation compatible with~$\bA$, hence $B\in\cF_n$.

\begin{align*}
R&=\setm{\vec x\in\bA^{n}}{f(\vec x)=c}\,, &&\text{by \eqref{E:Restpreimage}.}\\
&=\setm{\vec x\in\bA^{n}}{g(p_1(\vec x),\dots,p_{N+1}(\vec x))=c}\,, &&\text{by \eqref{E:factor}.}\\
&=\setm{\vec x\in\bA^{n}}{ (p_1(\vec x),\dots,p_{N+1}(\vec x))\in B}\,, &&\text{by definition of $B$.}
\end{align*}
As $p_1,\dots,p_n$ are terms in~$t$, it follows from Lemma~\ref{L:entailementFaciles}(2) that $\set{B,t}$ entails $R$, and so $\cR_{N+1}\cup \set{t}$ entails $R$.

Therefore $\cR_{N+1}\cup \set{t}$ entails the set of all completely meet-irreducible relation compatible with~$\bA$. It follows from Lemma~\ref{L:restricitoninfirreductibles} that $\cR_{N+1}\cup \set{t}$ entails $\cR$.
\end{proof}

\begin{theorem}
Let~$\bA$ be a finite Abelian algebra in a congruence-modular variety. Assume that $\card A=p_1^{\alpha_1}p_2^{\alpha_2}\dots p_k^{\alpha_k}$, where $p_1,\dots,p_k$ are distinct primes. Set $N=\max(4,1+\max_{1\le i\le k}(\alpha_i^3))$. Denote by $\cF$ the set of all $N$-ary relations compatible with~$\bA$. Then $\alg{A;\cF}$ dualizes~$\bA$.
\end{theorem}

\begin{proof}
Denote by $\cV$ the variety generated by~$\bA$. Let~$\bS$ be a subdirectly irreducible algebra in $\cV$. It follows from Corollary~\ref{C:CardinalSI} that $\card S$ divides $p_1^{\alpha_1^2}\dots p_k^{\alpha_k^2}$. Therefore by Corollary~\ref{C:cardinalHAS} $\bcH(\bA^2,\bS)$ has a generating family of size $\max_{1\le i\le k}(\alpha_i^3)$. Hence by Lemma~\ref{L:PrincipalEntailement}, it follows that $\cF\cup\set{t}$ entails all relations compatible with~$\bA$.

Note that $\gra t$ is a relation of arity $4$, hence by Lemma~\ref{L:entailementFaciles}(3) $\cF$ entails $\gra t$. It follows from Lemma~\ref{L:entailementFaciles}(4) that $\cF$ entails $t$. So $\cF$ entails all relations compatible with~$\bA$. Therefore it follows from Theorem~\ref{T:entaildualise} that $\alg{A;\cF}$ dualizes~$\bA$.
\end{proof}

\end{document}